\documentclass[reqno,12pt]{amsart}
\newcommand{\RR}{{\mathbb R}}
\newcommand{\CC}{{\mathbb C}}

\def\bege{\begin{equation}} \def\ende{\end{equation}}
\def\begr{\begin{eqnarray}} \def\endr{\end{eqnarray}}
\usepackage{amsmath}
\usepackage{amssymb}
\usepackage{cite}

\usepackage{bbm}
\usepackage{amsmath}
\usepackage{txfonts}
\usepackage{stmaryrd}
\usepackage{amssymb}
\usepackage{amsfonts}
\usepackage{times}
\usepackage{mathrsfs}

\usepackage{amsthm}
\usepackage{enumerate}

\usepackage{framed}
\usepackage{lipsum}
\usepackage{color}

\def\BB{ \mathbb{B}}
\def\SS{ \mathbb{S}}
\def\CC{ \mathbb{C}}
\newcommand{\DD}{{\mathbb D}}
\def\B{\mathcal{B}}
\def\R{\mathcal{R}}
\def\I{\mathcal{I}}

\def\D{\mathbb{D}}
\def\N{\mathbb N}

\def\hD{\hat{\mathcal{D}}}
\def\dD{\mathcal{D}}
\def\a{\alpha}

\def\vp{\varphi}
\def\om{\omega}

\def\begr{\begin{eqnarray}} \def\endr{\end{eqnarray}}

\def\msk{\medskip}
\def\ol{\overline}
\newtheorem{Lemma}{Lemma}
\newtheorem{Theorem}{Theorem}

\newtheorem{Remark}{Remark}

\begin{document}
\title[  ]{  Bergman projections induced by   doubling weights  on the unit ball of $\CC^n$}
  \author{ Juntao Du,    Songxiao Li$\dagger$, Xiaosong Liu and Yecheng Shi  }
  \address{Juntao Du\\ Faculty of Information Technology, Macau University of Science and Technology, Avenida Wai Long, Taipa, Macau.}
 \email{jtdu007@163.com  }
 \address{Songxiao Li\\ Institute of Fundamental and Frontier Sciences, University of Electronic Science and Technology of China,
  610054, Chengdu, Sichuan, P.R. China. } \email{jyulsx@163.com}

 \address{Xiaosong Liu \\ S Department of Mathematics, Shantou University, Shantou 515063,  P. R. China }\email{gdxsliu@163.com}

 \address{Yecheng Shi\\    School of Mathematics and Statistics, Lingnan Normal University,
      Zhanjiang 524048, Guangdong, P. R. China}\email{ 09ycshi@sina.cn}

 \begin{abstract} The boundedness of $P_\om:L^\infty(\BB)\to \B(\BB)$ and $P_\om(P_\om^+):L^p(\BB,\upsilon dV)\to L^p(\BB,\upsilon dV)$   on the unit ball of $\CC^n$ with $p>1$ and $\om,\upsilon\in\dD$ are investigated in this paper.
  \thanks{$\dagger$ Corresponding author.}
 \end{abstract}
 \maketitle

\section{Introduction}

Let $\BB$ be the open unit ball of $\CC^n$ and $\SS$   the boundary of $\BB$. When $n=1$,  $\BB$ is  the open unit disk in the complex plane $\mathbb{C}$ and always denoted by $\D$. Let $H(\BB)$ denote the space of all holomorphic functions on $\BB$. For any two points
$$z=(z_1,z_2,\cdots,z_n)\,\mbox{ and } \, w=(w_1,w_2,\cdots,w_n)$$
in $\CC^n$,  define
$\langle z,w \rangle=z_1\overline{w_1}+\cdots+z_n\overline{w_n}$ and
$$|z|=\sqrt{\langle z,z \rangle}=\sqrt{|z_1|^2+\cdots+|z_n|^2}.$$

Suppose $\om$ is a radial weight ( i.e., $\om$ is a positive, measurable and  integrable  function on $[0,1)$ and $\om(z)=\om(|z|)$ for all $z\in\BB$).  Let $\hat{\om}(r)=\int_r^1\om(t)dt$.
We say that
\begin{itemize}
  \item $\om$ is a doubling weight, denoted by $\om\in \hD$,  if there is a constant $C>0$ such that
$$\hat{\om}(r)<C\hat{\om}(\frac{1+r}{2}) ,\,\,\mbox{ when } 0\leq r<1;$$
  \item  $\om$ is a regular weight,  denoted by  $\om\in    \R $, if there exist $C>0$ and $\delta\in(0,1)$ such that
$$\frac{1}{C}<\frac{\hat{\om}(t)}{(1-t)\om(t)}<C,\,\mbox{ when }\, t\in(\delta,1);$$
 \item  $\om$ is a rapidly increasing weight, denoted by  $\om\in\I$, if (see \cite{PjaRj2014book})
$$\lim_{r\to 1} \frac{\hat{\om}(r)}{(1-r)\om(r)}=\infty;$$
  \item $\om$ is a reverse doubling weight, denoted by $\om\in\check{\dD}$, if
 there exist $K>1$ and $C>1$, such that
\begin{align}\label{0420-1}
\hat{\om}(t)\geq C\hat{\om}(1-\frac{1-t}{K}),\,\,\,\,t\in(0,1).
\end{align}
\end{itemize}
   See  \cite{Pja2015,PjaRj2014book} and the references therein for more details about $\I,\R$,  $\hD$.
Let $\dD=\hD\cap\check{\dD}$. More information about $\check{\dD}$ and $\dD$ can be found in \cite{PjRj2019arxiv,KtPjRj2018arxiv}.

 Let $d\sigma$  and $dV$ be the normalized surface and volume measures on $\SS$ and $\BB$,  respectively.
For $0<p< \infty$, the  Hardy space $H^p(\BB) $(or $H^p$) is the space consisting of all functions $f\in H(\BB)$ such that
$$\|f\|_{H^p}:=\sup_{0<r<1}M_p(r,f)<\infty,$$
  where
$$M_p(r,f):=\left(\int_{\SS}|f(r\xi)|^pd\sigma(\xi)\right)^\frac{1}{p}, \,\,\,\,0<p<\infty.$$
 $H^\infty$ is the space consisting of all $f\in H(\BB)$ such that $$\|f\|_{H^\infty}:=\sup_{z\in\BB}|f(z)|<\infty.$$

For any $f\in H(\BB)$, let $\Re f$ be the radial derivative of $f$,  that is,
$$\Re f(z)=\sum_{k=1}^n z_k\frac{\partial f}{\partial z_k}(z),\,\,\,\,z=(z_1,z_2,\cdots,z_n)\in\BB.$$
Then Bloch space $\B(\BB)$ consist of all $f\in H(\BB)$ such that
$$ \|f\|_{\B(\BB)}:=|f(0)|+\sup_{z\in\BB}(1-|z|^2)|\Re f(z)|<\infty.$$
When $n=1$, $\|\cdot\|_{\B(\D)}$ is a little difference from the norm defined in   classical way, see \cite{zhu} for example, but they are equivalent.
We also keep $\B$ as the abbreviation of  $\B(\BB)$.

Suppose $\mu$ is  a positive Borel measure on $\BB$ and $0<p<\infty$. The Lebesgue space $L^p(\BB,d\mu)$ consists of all measurable
complex functions $f$ on $\BB$ such that $|f|^p$ is integrable with respect to $\mu$, that is, $f\in L^p(\BB,d\mu)$ if and only if
$$\|f\|_{L^p(\BB,d\mu)}:=\left(\int_\BB |f(z)|^pd\mu(z)\right)^\frac{1}{p}<\infty.$$
$L^\infty(\BB,d\mu)$  consists of all measurable complex functions $f$ on $\BB$ such that $f$ is essential bounded, that is, $f\in L^\infty(\BB,d\mu)$ if and only if
$$\|f\|_{L^\infty(\BB, d\mu)}:=\inf_{E\subset\BB,\mu(E)=0}\sup_{z\in\BB\backslash E}|f(z)|<\infty.$$
More details about $L^p(\BB,d\mu)$ can be seen in \cite{Rw1980,Zk2005}.
For a positive and measurable function $\om$ on $\BB$, if  $\om\in L^1(\BB,dV)$,  letting $d\mu(z)=\om(z)dV(z)$,  $\mu$ is a Borel measure on $\BB$.
 Then, we will write $L^p(\BB,d\mu)$ as $L^p(\BB,\om dV)$.
When $n=1$ and $z\in\D$, let $dV(z)=\frac{1}{\pi}dA(z)$ be the normalized area measure on $\D$. Then we can define
the Lebesgue space on the unit disk in the same way.

In \cite{PjaRj2014book}, J. Pel\'aez and J. R\"atty\"a introduced a new class of  weighted Bergman spaces, denoted by  $A_\om^p(\D)$, which are induced by  a rapidly increasing weight $\om$ in $\D$. That is
$$A_\om^p(\DD)=L^p(\D,\om dA)\cap H(\D),\,\mbox{where }\,0<p<\infty.$$
See \cite{Pja2015,PjaRj2014book,PjaRj2015,PjaRj2016,PjRj2016jmpa,PjaRjSk2015mz,PjaRjSk2018jga} for more results on $A_\om^p(\D)$ with $\om\in\hD$.  In \cite{DjLsLxSy2019arxiv}, we extended the Bergman space $A_\om^p(\D)$  to the unit ball $\BB$ of $\CC^n$.   That is,
$$A_\om^p(\BB)=L^p(\BB,\om dA)\cap H(\BB),\,\mbox{where }\,0<p<\infty.$$
For brief, let $A_\om^p=A_\om^p(\BB)$.
As a subspace of $L^p(\BB,\om dV)$, the norm on $A_\om^p$ will be written as $\|\cdot\|_{A_\om^p}$.
 It is easy to check that $A_\om^p$ is a Banach space when $p\geq 1$ and a complete metric space with the distance $\rho(f,g)=\|f-g\|_{A_\om^p}^p$ when $0<p<1$.
When $\alpha>-1$  and $c_\a=\frac{\Gamma(n+\a+1)}{\Gamma(n+1)\Gamma(\a+1)}$, if $\om(z)=c_\alpha(1-|z|^2)^\alpha$, the space $A_\om^p$ becomes the classical weighted Bergman space $A_\alpha^p$, and we write $dV_\alpha(z)=c_\alpha(1-|z|^2)^\alpha dV(z)$.
 When $\alpha=0$, $A^p_0=A^p$ is the standard Bergman space.
See \cite{Rw1980,Zk2005} for the theory of $H^p$ and $A_\alpha^p$.

When $p=2$, the space $A_\om^2$ is a Hilbert space with the inner product
$$\langle f,g \rangle_{A_\om^2}=\int_\BB f(z)\ol{g(z)}\om(z)dV(z), \,\,\mbox{ for all }\,\,f,g\in A_\om^2.$$
In a standard way, for every $z\in\BB$, the point evaluation $L_z f=f(z)$  is a bounded linear functional on  $A_\om^2$.
Then, the  Riesz representation theory  shows that, there exists a unique function $B_z^\om$
such that
\begin{align*}
f(z)=\langle f, B_z^\om \rangle_{A_\om^2}=\int_\BB f(w)\ol{B_z^\om(w)}\om(w)dV(w),\,\,\mbox{ for all }\,\,f\in A_\om^2.
\end{align*}
So, for any $f\in L^1(\BB,\om dV)$,   the Bergman projection $P_\om f$ is defined by
\begin{align*}
P_\om f(z)=\int_\BB f(\xi)\ol{B_z^\om(\xi)}\om(\xi)dV(\xi),
\end{align*}
while the  maximal Bergman projection $P_\om^+$ is defined by
$$P_\om^+(f)(z)=\int_\BB f(\xi)\left|B_z^\om(\xi)\right|\om(\xi)dV(\xi).   $$
When $\om(z)=c_\alpha(1-|z|^2)^\alpha(\alpha>-1)$,  $P_\om$ and $P_\om^+$ are written as $P_\alpha$ and $P_\alpha^+$, respectively.

The study about the Bergman projection has a long history and important sense.
If $s\in\CC$ such that  $\mathrm{Re}s>-1$, for all  $f\in L^1(\BB,dV)$, let
\begin{align*}
T_s f (z)=\frac{\Gamma(n+s+1)}{\Gamma(n+1)\Gamma(s+1)}\int_{\BB}\frac{(1-|w|^2)^s}{(1-\langle z,w\rangle)^{n+1+s}}f(w)dV(w).
\end{align*}
Obviously, when $s=0$, we have $T_0=P_0$.
In \cite{FfRw1974iumj}, Forelli and  Rudin proved that if $1\leq p<\infty $, $T_s$ is bounded on $L^p(\BB,dV)$ if and only if  $(1+\mathrm{Re}s)p>1$.
In \cite{Cbr1990pams},   Choe proved that, if $p\geq 1$ and $\alpha>-1$,  $T_s$ is bounded on $L^p(\BB, dV_\alpha)$ if and only if
$(1+\mathrm{Re}s)p>1+\alpha.$
In \cite{Lc2015jfa}, Liu gave a sharp estimate about the norm of $P_0$ on $L^p(\BB,dV)$.
  From \cite[Theorem 2.11]{Zk2005}, we see that $P_\alpha$ is  bounded on $L^p(\BB,dV_\beta)$  if and only if
$p(\alpha+1)>\beta+1$ when $p\geq 1, \alpha,\beta\in(-1,\infty)$.

In the setting of the unit disk, Bekoll\'e and  Bonami   showed that, if $1<p<\infty$, $\upsilon$ is positive on $\D$ and $\int_\D \upsilon(z)dA_\alpha(z)<\infty$,
$P_\alpha: L^p(\D, \upsilon dA_\alpha)\to L^p(\D, \upsilon dA_\alpha)$ is bounded if and only if $\upsilon$ satisfies the Bekoll\'e-Bonami condition, see\cite{BdBa1978,Bd1981}.
  The result was extended in \cite{PjRjWbarxiv2016} for some $\om\in\R$.
In \cite{PjRj2016jmpa},  for $\om\in\R$, the Bergman projections $P_\om$ and the  maximal Bergman projection  $P_\om^+$  on some analytic function spaces on   $\D$ were studied.
  In \cite{PjRj2019arxiv}, Pel\'aez and R\"atty\"a studied the Bergman projection induced by radial weight on some analytic function spaces on $\D$.   They showed that
\begin{itemize}
  \item $P_\om:L^\infty(\D,dA)\to \B(\D)$ is  bounded if and only if $\om\in\hD$;
  \item $P_\om:L^\infty(\D,dA)\to \B(\D)$ is  bounded and onto if and only if $\om\in\dD$;
  \item If $\om\in \hD$ and $p>1$, $P_\om^+:L^p(\D,\om dA)\to L^p(\D,\om dA)$ is bounded if and only if $\om\in\dD$;
  \item If $1<p<\infty$, $\om\in\hD$, $\upsilon$ is a radial weight, then the boundedness of $P_\om^+:L^p(\D,\upsilon dA)\to L^p(\D,\upsilon dA)$ implies $\om,\upsilon\in \dD$.
\end{itemize}

Motivated by \cite{PjRj2016jmpa, PjRj2019arxiv}, in this paper we investigate the boundedness of $P_\om:L^\infty(\BB,dV)\to \B(\BB)$ and $P_\om(P_\om^+):L^p(\BB,\upsilon dV)\to L^p(\BB,\upsilon dV)$ on the unit ball of $\CC^n$ with $p>1$ and $\om,\upsilon\in\dD$.

This paper is organized as follows.  In Section 2, we recall some  results and notations.
In Section 3, we give some estimates for $B_z^\om$ with $\om\in\hD$.
In Section 4, we investigate the boundedness of  $P_\om$ and $P_\om^+$ with $\om\in\dD$.

Throughout this paper,  the letter $C$ will denote a constant  and may differ from one occurrence to the other.
The notation $A \lesssim B$ means that there is a positive constant $C$ such that $A\leq CB$.
The notation $A \approx B$ means $A\lesssim B$ and $B\lesssim A$. \msk

 \section{ Preliminary results}

 For any $\xi,\tau\in\overline{\BB}$, let $d(\xi,\tau)=|1-\langle \xi,\tau\rangle|^\frac{1}{2}$.
 Then $d(\cdot,\cdot)$ is a nonisotropic metric.
For $r>0$   and $\xi\in\SS$, let
  $$Q(\xi,r)=\{\eta\in \SS:|1-\langle\xi,\eta\rangle|\leq r^2\}.$$
 $Q(\xi,r)$  is a ball in $\SS$ for all $\xi\in \SS$ and $r\in(0,1)$.
More information about $d(\cdot,\cdot)$ and $Q(\xi,r)$ can be found in \cite{Rw1980,Zk2005}.

For any $a\in\BB\backslash\{0\}$, let
$Q_a=Q({a}/{|a|},\sqrt{1-|a|})$
and
$$S_a=S(Q_{a})=\left\{z\in\BB:\frac{z}{|z|}\in Q_{a},|a|<|z|<1\right\}.$$
For convince, if $a=0$, let $Q_a=\SS$ and $S_a=\BB$.  We call $S_a$ the Carleson block.
 See \cite{DjLsLxSy2019arxiv} for more information about the Carleson block.
As usual, for a measurable set $E\subset\BB$, $\om(E)=\int_E \om(z)dV(z)$.\msk

\begin{Lemma}\label{0507-1}
Suppose $\om$ is a radial weight. Then
\begin{enumerate}[(i)]
  \item The following statements are equivalent.
  \begin{enumerate}[(a)]
    \item  $\om\in\hD$;
    \item there is a constant $b>0$ such that $\frac{\hat{\om}(t)}{(1-t)^b}$ is essentially increasing;
    \item for all $x\geq 1$, $\int_0^1 s^x\om(s)ds\approx \hat{\om}(1-\frac{1}{x})$.
  \end{enumerate}
  \item $\om\in\check{\dD}$ if and only if there is a constant $a>0$ such that $\frac{\hat{\om}(t)}{(1-t)^a}$ is essentially decreasing.
  \item If $\om$ is continuous, then $\om\in \R$ if and only if  there are  $-1<a<b<+\infty$ and $\delta\in [0,1)$, such that
  \begin{align}\label{0515-1}\frac{{\om}(t)}{(1-t)^b} \nearrow\infty,
\,\,\mbox{ and }\,\,\frac{{\om}(t)}{(1-t)^a}\searrow 0,\,\,\mbox{ when }\,\,\delta\leq t<1 .
\end{align}
\end{enumerate}
\end{Lemma}
Lemma \ref{0507-1} plays an important role in this research and  can be found in many papers.
Here,  we refer to Lemmas B and C in \cite{KtPjRj2018arxiv} and observation $(v)$ of Lemma 1.1 in \cite{PjaRj2014book}.

For any radial weight $\om$,  its associated weight $\om^*$  is defined by
\begin{align*}\om^*(z)=\int_{|z|}^1 \om(s)\log\frac{s}{|z|}sds, \,\,z\in\D\backslash\{0\}.\end{align*}
The following lemma gives some properties and applications of $\om^*$.
\begin{Lemma}\label{1210-3}
Assume that $\om\in\hD$. Then  the following statements hold.
\begin{enumerate}[(i)]
  \item  $\om^*(r)\approx (1-r)\int_r^1 \om(t)dt$ when $r\in(\frac{1}{2},1)$.
    \item For any $\alpha>-2$, $(1-t)^\alpha\om^*(t)\in\R$.
\item  $\om(S_a)\approx (1-|a|)^n\int_{|a|}^1 \om(r)dr$.
\item $\hat{\om}(z)\approx\hat{\om}(a)$, if   $\frac{1}{C}<\frac{1-|z|}{1-|a|}<C$ for some fixed $C>1$.
\end{enumerate}
\end{Lemma}

\begin{proof} $(i)$ and $(ii)$  are Lemmas 1.6 and  1.7 in \cite{PjaRj2014book},  respectively. $(iii)$ was proved in \cite{DjLsLxSy2019arxiv}.
$(iv)$ can be proved straightly by $(i), (ii)$ and Lemma \ref{0507-1}. For the benefit of readers, we give a detailed proof.

Suppose $\om\in\hD$. Then there exist $a,b>-1$ and $\delta\in(0,1)$ such that (\ref{0515-1}) holds for $\om^*$. Then, for all $\delta<x\leq y<1$ such that $\frac{1}{C}\leq\frac{1-x}{1-y}\leq C$,  we have
\begin{align*}
 1\approx \left( \frac{1-x}{1-y}\right)^a
\leq \frac{\om^*(x)}{\om^*(y)}
\leq \left( \frac{1-x}{1-y}\right)^b
\approx 1.
\end{align*}
If  $x\leq \delta$ and $\frac{1}{C}\leq\frac{1-x}{1-y}\leq C$, $\hat{\om}(x)\approx\hat{\om}(y)$ is obvious. So, we have
\begin{align*}
\hat{\om}(z)\approx\hat{\om}(a),\,\mbox{ if }\, \frac{1}{C}<\frac{1-|z|}{1-|a|}<C.
\end{align*}
The proof is complete.
\end{proof}

For a Banach space or a complete metric space $X$  and a positive Borel measure $\mu$ on $\BB$, $\mu$ is a  $q-$Carleson measure for $X$ means that the identity operator $Id:X\to L_\mu^q$ is bounded. When $0<p\leq q<\infty$ and $\om\in\hD$, a characterization of $q-$Carleson measure for $A_\om^p$ was obtained in \cite{DjLsLxSy2019arxiv}.
\msk

\noindent{\bf Theorem A. }{\it
Let $0<p\leq q<\infty$, $\om\in \hD$, and $\mu$ be a positive Borel measure on $\BB$. Then  $\mu$ is a $q$-Carleson measure for $A_\om^p$ if and only if
  \begin{align}\label{0109-8}
  \sup_{a\in\BB} \frac{\mu(S_a)}{(\om(S_a))^{\frac{q}{p}}}<\infty.
  \end{align}
  Moreover, if $\mu$  is a $q$-Carleson measure for $A_\om^p$, then the identity operator $Id:A_\om^p\to L_\mu^q$ satisfies
  $$\|Id\|_{A_\om^p\to L_\mu^q}^q \approx \sup_{a\in\BB} \frac{\mu(S_a)}{(\om(S_a))^{\frac{q}{p}}}.$$
}

\section{Some estimates about $B_z^\om$ with $\om\in\hD$}
In this section, we consider the reproducing kernel of $A_\om^2$ and give some estimates for it. First, let's recall some notations.
For all $f\in H(\BB)$, the Taylor series of $f$ at origin, which converges absolutely and uniformly on each compact subset of $\BB$,  is
$$f(z)=\sum_{m}a_mz^m, \,\,z\in\BB.$$
Here the summation is over all multi-indexs
$m=(m_1,m_2,\cdots,m_n),$ where each $m_k$ is a nonnegative integer and
$z^m=z_1^{m_1}z_2^{m_2}\cdots z_{n}^{m_n}.$  Let $|m|=m_1+m_2+\cdots+m_n, \,\, m!=m_1!m_2!\cdots m_n!,$ and
$f_k(z)=\sum_{|m|=k }a_mz^m.$  Then the Taylor series of $f$ can be written as
$$f(z)=\sum_{k=0}^\infty f_k(z),$$
which is called the homogeneous expansion of $f$.

\begin{Lemma}\label{0313-1} Suppose $\om\in\hD$. Then,
$$B_z^\om(w)=\frac{1}{2n!}\sum_{k=0}^\infty \frac{(n-1+k)!}{k!\om_{2n+2k-1}} \langle w,z\rangle^k,$$
and
$$\|B_z^\om\|_{\B}\approx\frac{1}{\om(S_z)}\approx \|B_z^\om\|_{H^\infty},\,\,z\in\BB.$$
Here and henceforth, $\om_s=\int_0^1 r^s\om(r)dr$.
\end{Lemma}

\begin{proof}
Suppose $f\in A_\om^2$ and
$$f(z)=\sum_{m}a_mz^m, \,\,z\in\BB.$$
For any fixed $z\in\BB$, let $$B_z^\om(w)=\sum_{m}b_m(z)w^m.$$
By Lemmas 1.8 and 1.11  in \cite{Zk2005}, we have
\begin{align*}
f(z)
&=\int_\BB f(w)\ol{B_z^\om(w)}\om(w)dV(w)  \\
&=2n\sum_{m} \frac{(n-1)!m!}{(n-1+|m|)!} a_m\ol{b_m(z)}\int_0^1 r^{2n+2|m|-1} \om(r)dr  \\
&=2n!\sum_{m} \frac{m!}{(n-1+|m|)!} a_m\ol{b_m(z)}\om_{2n+2|m|-1}.
\end{align*}
Since $f$ is arbitrary,
$$z^m= \frac{2 n!m!}{(n-1+|m|)!} \om_{2n+2|m|-1}\ol{b_m(z)}.$$
Therefore, we have
\begin{align*}
B_z^\om(w)
&=\frac{1}{2 n!}\sum_{k=0}^\infty \frac{(n-1+k)!}{k!  \om_{2n+2k-1}} \sum_{|m|=k} \frac{|m|!}{m!} w^m\ol{z}^m \\
&=\frac{1}{2n!}\sum_{k=0}^\infty \frac{(n-1+k)!}{k!\om_{2n+2k-1}} \langle w,z\rangle^k.
\end{align*}
Then,
\begin{align*}
\Re B_z^\om(w)
=\frac{1}{2n!}\sum_{k=1}^\infty \frac{(n-1+k)!}{(k-1)!\om_{2n+2k-1}} \langle w,z\rangle^k.
\end{align*}
By Stirling estimate and Lemma \ref{0507-1}, when $\frac{1}{2}\leq |z|<1$, we have
$$
|B_z^\om(w)| \lesssim \sum_{k=1}^\infty \frac{k^{n-1}|z|^{k}}{\om_{2n+2k-1}}
\approx \sum_{k=n}^\infty \frac{(k+1)^{n-1}|z|^{k}}{\om_{2k+1}},
$$
and
\begin{align*}
|\Re B_z^\om(z)|
\approx\sum_{k=1}^\infty \frac{k^n|z|^{2k}}{\om_{2n+2k-1}}
\approx \sum_{k=n+1}^\infty \frac{(k+1)^n|z|^{2k}}{\om_{2k+1}}.
\end{align*}
Let $\hat{\om}_\alpha(t)=(1-t)^\alpha\hat{\om}(t)$ for any fixed $\alpha\in \RR$.
Using (20) in \cite{PjRj2016jmpa} and Lemma \ref{1210-3}, we have
$$
\sum_{k=n}^\infty \frac{(k+1)^{n-1}|z|^{k}}{\om_{2k+1}}
\approx \int_0^{|z|}\frac{1}{\hat{\om}_{n+1}(t)}dt
\lesssim \frac{1}{(1-|z|)^n\hat{\om}(z)}
\approx \frac{1}{\om(S_z)},
$$
and
$$
\sum_{k=n}^\infty \frac{(k+1)^n|z|^{2k}}{\om_{2k+1}}
\approx \int_0^{|z|^2}\frac{1}{(1-t)^{n+2}\hat{\om}(t)}dt.
$$
By Lemma \ref{0507-1}, there exists a constant $b>0$ such that $\frac{\hat{\om}(t)}{(1-t)^b}$ is essentially increasing.
So, by Lemma \ref{1210-3},
$$
\int_0^{|z|^2}\frac{1}{(1-t)^{n+2}\hat{\om}(t)}dt
\gtrsim \frac{(1-|z|^2)^b}{\hat{\om}(|z|^2)} \int_0^{|z|^2}\frac{1}{(1-t)^{n+2+b}}dt
\approx \frac{1}{(1-|z|)\om(S_z)}.
$$
Therefore, when $\frac{1}{2}\leq |z|<1$, we have
\begin{align}\label{0506-1}
\|B_z^\om\|_{H^\infty}
\lesssim \frac{1}{\om(S_z)},
\,\,\,\,\,
\frac{1}{\om(S_z)}\lesssim \|B_z^\om\|_\B.
\end{align}
When $|z|<\frac{1}{2}$, since $\om(S_z)\approx 1$,  $\|B_z^\om\|_{\B}\geq |B_z^\om(0)|\gtrsim 1$, and
$$|B_z^\om(w)|\leq \frac{1}{2n!}\sum_{k=0}^\infty \frac{(n-1+k)!}{k!\om_{2n+2k-1}}\frac{1}{2^k}<\infty,$$
(\ref{0506-1}) also holds.
By the well known fact that $\|f\|_{\B}\lesssim \|f\|_{H^\infty}$, we obtain the desired result.  The proof is complete.
\end{proof}

\begin{Lemma}\label{0313-2}
Let $0<p<\infty$, $\om\in\hD$. Then the following assertions hold.
\begin{enumerate}[(i)]
  \item When $|rz|>\frac{1}{4}$, we have
     $$M_p^p(r, B_z^\om) \approx \int_{0}^{r|z|} \frac{1}{\hat{\om}(t)^p (1-t)^{np-n+1}}dt,$$
     and
     $$M_p^p(r, \Re B_z^\om) \approx \int_{0}^{r|z|} \frac{1}{\hat{\om}(t)^p (1-t)^{(n+1)p-n+1}}dt.$$
  \item If $\upsilon\in\hD$, when $|z|>\frac{6}{7}$, we have
  $$\|B_z^\om\|_{A_\upsilon^p}^p \approx   \int_{0}^{|z|} \frac{\hat{\upsilon}(t)}{\hat{\om}(t)^p (1-t)^{np-n+1}}dt,$$
  and
  $$  \|\Re B_z^\om\|_{A_\upsilon^p}^p\approx  \int_0^{|z|}  \frac{\hat{\upsilon}(t)}{\hat{\om}(t)^p (1-t)^{(n+1)p-n+1}}dt.$$
\end{enumerate}
\end{Lemma}
\begin{proof}
When $n=1$, the theorem was proved in \cite{PjRj2016jmpa}, so we always assume $n\geq 2$.
Since we will use some results on $A_\om^p(\D)$, for brief, the symbol $A_\om^p$ only means $A_\om^p(\BB)$ with $n\geq 2$.
Meanwhile,  let $B_z^{\om,1}$ denote the reproducing kernel of $A_\om^2(\D)$.
Recall that, on the unit disk, $dA_\alpha(z)=c_\alpha(1-|z|^2)^\alpha dA(z)$, where $dA(z)$ is the normalized area measure on $\D$.

By Lemma \ref{0313-1}, we have
$$\Re B_z^\om (w)=\frac{1}{2n!}\sum_{k=1}^\infty \frac{(n-1+k)!}{(k-1)!\om_{2n+2k-1}}\langle w,z\rangle^k. $$
Let $e_1=(1,0,\cdots,0)$.
When $|rz|>0$, by rotation transformation and Lemma 1.9 in \cite{Zk2005},  we have
\begin{align*}
M_p^p(r, \Re B_z^\om)
&=  M_p^p(r, \Re B_{|z|e_1}^\om)
=\frac{1}{2n!}\int_{\SS}\left|\sum_{k=1}^\infty \frac{(n-1+k)!}{(k-1)!\om_{2n+2k-1}}\langle r\eta,|z|e_1\rangle^k\right|^p d\sigma(\eta)  \\
&\approx \int_{\D} \left|\sum_{k=1}^\infty \frac{(n-1+k)! }{(k-1)!\om_{2n+2k-1}}(r|z|\xi)^k\right|^p (1-|\xi|^2)^{n-2}dA(\xi)  \\
&=\frac{1}{|rz|^{(n-1)p}}  \int_{\D} \left|\sum_{k=1}^\infty \frac{|rz|^{n+k-1}(\xi^{n+k-1})^{(n)} }{\om_{2(n+k-1)+1}}\right|^p |\xi|^pdA_{n-2}(\xi)  \\
&\approx \frac{1}{|rz|^{(n-1)p}}  \int_{\D} \left|\sum_{k=0}^\infty \frac{|rz|^{k}(\xi^{k})^{(n)} }{\om_{2k+1}}\right|^p dA_{n-2}(\xi)  \\
&=\frac{1}{|rz|^{(n-1)p}}  \|(B_{r|z|}^{\om,1})^{(n)}\|_{A_{n-2}^p(\D)}^p.
\end{align*}
When $r|z|>\frac{1}{4}$, by Theorem 1 in \cite{PjRj2016jmpa}, we have
\begin{align*}
M_p^p(r, \Re B_z^\om)
\approx \int_{0}^{r|z|} \frac{(1-t)^{n-1}}{\hat{\om}(t)^p (1-t)^{p(n+1)}}dt
=\int_{0}^{r|z|} \frac{1}{\hat{\om}(t)^p (1-t)^{p(n+1)-n+1}}dt.
\end{align*}
Therefore, when $|z|>\frac{6}{7}$, by Fubini's theorem, we have
\begin{align*}
\|\Re B_z^\om\|_{A_\upsilon^p}^p
&\approx \int_{\frac{1}{2}}^1 r^{2n-1}\upsilon(r)M_p^p(r,\Re B_z^\om)dr
=\int_0^{|z|}  \frac{\int_{\max\{\frac{t}{|z|},\frac{1}{2}\}}^1 r^{2n-1}\upsilon(r)dr}{\hat{\om}(t)^p (1-t)^{p(n+1)-n+1}} dt.
\end{align*}

When $0\leq t\leq \frac{|z|}{2}$, we have
\begin{align}\label{0522-1}\int_{\max\{\frac{t}{|z|},\frac{1}{2}\}}^1 r^{2n-1}\upsilon(r)dr
=\int_\frac{1}{2}^1 r^{2n-1}\upsilon(r)dr\approx 1\approx\hat{\upsilon}(t).
\end{align}

When $\frac{|z|}{2}\leq t\leq |z|$, we have
\begin{align}\label{0522-2}
\int_{\max\{\frac{t}{|z|},\frac{1}{2}\}}^1 r^{2n-1}\upsilon(r)dr
=\int_\frac{t}{|z|}^1 r^{2n-1}\upsilon(r)dr
\leq  \hat{\upsilon}(t).
\end{align}
By $\upsilon\in\hD$ and Lemma \ref{0507-1}, there exists a constant $b>0$ such that $\frac{\hat{\upsilon}(t)}{(1-t)^b}$ is essentially increasing. So,
\begin{align}
\int_\frac{|z|}{2}^{|z|}  \frac{\int_{\max\{\frac{t}{|z|},\frac{1}{2}\}}^1 r^{2n-1}\upsilon(r)dr}{\hat{\om}(t)^p (1-t)^{p(n+1)-n+1}} dt
\gtrsim & \int_\frac{|z|}{2}^{2|z|-1}  \frac{\hat{\upsilon}(\frac{t}{|z|})}{\hat{\om}(t)^p (1-t)^{p(n+1)-n+1}} dt  \nonumber \\
\gtrsim &\int_\frac{|z|}{2}^{2|z|-1}  \frac{\hat{\upsilon}(t)}{\hat{\om}(t)^p (1-t)^{p(n+1)-n+1}} \left(\frac{1-\frac{t}{|z|}}{1-t}\right)^b dt \nonumber\\
\gtrsim & \int_\frac{|z|}{2}^{2|z|-1}  \frac{\hat{\upsilon}(t)}{\hat{\om}(t)^p (1-t)^{p(n+1)-n+1}}  dt,\label{0522-3}
\end{align}
where the last  estimate follows by
$$\frac{1}{|z|}\frac{|z|-t}{1-t}\geq \frac{1}{|z|}\frac{|z|-(2|z|-1)}{1-(2|z|-1)}\gtrsim 1,\,\mbox{ for all }\,t\in(\frac{|z|}{2},2|z|-1)\,\mbox{ and }\,|z|>\frac{6}{7}.$$
Meanwhile, by $\om,\upsilon\in\hD$ and Lemma \ref{1210-3}, we have
\begin{align}
\int_{\frac{|z|}{2}}^{2|z|-1}  \frac{\hat{\upsilon}(t)}{\hat{\om}(t)^p (1-t)^{p(n+1)-n+1}}dt
&\geq \int_{4|z|-3}^{2|z|-1}  \frac{\hat{\upsilon}(t)}{\hat{\om}(t)^p (1-t)^{p(n+1)-n+1}}dt \nonumber\\
&\approx \frac{\hat{\upsilon}(2|z|-1)}{\hat{\om}(2|z|-1)^p (1-|z|)^{p(n+1)-n}} \nonumber\\
&\approx \int_{2|z|-1}^{|z|}  \frac{\hat{\upsilon}(t)}{\hat{\om}(t)^p (1-t)^{p(n+1)-n+1}}dt. \label{0522-4}
\end{align}
By (\ref{0522-3}) and (\ref{0522-4}), we have
\begin{align}
\int_\frac{|z|}{2}^{|z|}  \frac{\int_{\max\{\frac{t}{|z|},\frac{1}{2}\}}^1 r^{2n-1}\upsilon(r)dr}{\hat{\om}(t)^p (1-t)^{p(n+1)-n+1}} dt
\gtrsim&2\int_\frac{|z|}{2}^{2|z|-1}  \frac{\hat{\upsilon}(t)dt}{\hat{\om}(t)^p (1-t)^{p(n+1)-n+1}} dt     \nonumber\\
\gtrsim&\left(\int_\frac{|z|}{2}^{2|z|-1}+\int_{2|z|-1}^{|z|}\right)  \frac{\hat{\upsilon}(t)dt}{\hat{\om}(t)^p (1-t)^{p(n+1)-n+1}} dt     \nonumber\\
= & \int_\frac{|z|}{2}^{|z|}  \frac{\hat{\upsilon}(t)}{\hat{\om}(t)^p (1-t)^{p(n+1)-n+1}}  dt.   \label{0522-5}
\end{align}

So, if $|z|>\frac{6}{7}$, by (\ref{0522-1}) and (\ref{0522-2}), we have
\begin{align*}
\int_0^{|z|}  \frac{\int_{\max\{\frac{t}{|z|},\frac{1}{2}\}}^1 r^{2n-1}\upsilon(r)dr}{\hat{\om}(t)^p (1-t)^{p(n+1)-n+1}} dt
\lesssim  \int_0^{|z|}  \frac{\hat{\upsilon}(t)}{\hat{\om}(t)^p (1-t)^{p(n+1)-n+1}}  dt,
\end{align*}
and by (\ref{0522-1}) and (\ref{0522-5}), we get
\begin{align*}
\int_0^{|z|}  \frac{\int_{\max\{\frac{t}{|z|},\frac{1}{2}\}}^1 r^{2n-1}\upsilon(r)dr}{\hat{\om}(t)^p (1-t)^{p(n+1)-n+1}} dt
&\gtrsim  \int_0^{|z|}  \frac{\hat{\upsilon}(t)}{\hat{\om}(t)^p (1-t)^{p(n+1)-n+1}}  dt.
\end{align*}
Therefore,
\begin{align*}
\|\Re B_z^\om\|_{A_\upsilon^p}^p
&\approx \int_0^{|z|}  \frac{\hat{\upsilon}(t)}{\hat{\om}(t)^p (1-t)^{p(n+1)-n+1}}dt.
\end{align*}
The rest of the lemma can be proved in the same way. The proof is complete.
\end{proof}

\section{Main results and proofs}
In this section, we give main results and proofs in this paper.
We should note that, if $\om\in\hD$, we have
$$\|f\|_{L^\infty(\BB,\om dV)}=\|f\|_{L^\infty(\BB,dV)}.$$
So, let  $L^\infty= L^\infty(\BB,\om dV)=L^\infty(\BB,dV)$ in this section.

\begin{Theorem}\label{0315-1} When  $\om\in\dD$,  $P_\om:L^\infty\to\B$ is bounded and  onto.
\end{Theorem}

\begin{proof}  For all $f\in L^\infty$, by Lemma \ref{0313-2},  we have
\begin{align*}
|\Re (P_\om f)(z)|
\leq \int_\BB |f(w)||\Re B_z^\om(w)|\om(w)dV(w)
\leq \|f\|_{L^\infty}\|\Re B_z^\om\|_{A_\om^1}
\lesssim \frac{\|f\|_{L^\infty}}{1-|z|}.
\end{align*}
So, $P_\om:L^\infty\to\B$ is bounded.

 By (19) in \cite{DjLsLxSy2019arxiv},  we have
$$
\|f\|_{A_\om^2}^2 =\om(\BB)|f(0)|^2+ 4  \int_{\BB}\frac{|\Re f(z)|^2 }{|z|^{2n}}  \om^{n*}(z)dV(z),
$$
where $$\om^{n*}(z):=\int_{|z|}^1 r^{2n-1}\log\frac{r}{|z|}\om(r)dr.$$
So, for $f,g\in A_\om^2$, we have
$$
\langle f,g\rangle_{A_\om^2}=\om(\BB)f(0)\ol{g(0)}
  + 4\int_{\BB} \frac{\Re f(z)\ol{\Re g(z)}}{|z|^{2n}}\om^{n*}(z)dV(z).
$$
Let $$W_1(t):=\frac{\hat{\om}(t)}{1-t},\,\,\mbox{ and }\,\,W_1(z):=W_1(|z|).$$
Since $\om\in\dD$, by Lemma \ref{0507-1}, there are constants $a,b>0$ such that
$\frac{\hat{\om}(t)}{(1-t)^a}$ is essentially decreasing and $\frac{\hat{\om}(t)}{(1-t)^b}$ is essentially increasing.
Thus, we have
$$\int_r^1 \frac{\hat{\om}(t)}{1-t}dt\lesssim \frac{\hat{\om}(r)}{(1-r)^a}\int_r^1 (1-t)^{a-1}dt\approx \hat{\om}(r),$$
and
$$\int_r^1 \frac{\hat{\om}(t)}{1-t}dt\gtrsim \frac{\hat{\om}(r)}{(1-r)^b}\int_r^1 (1-t)^{b-1}dt\approx \hat{\om}(r).$$
Then,
$$
\hat{W_1}(r)=\int_r^1 \frac{\hat{\om}(t)}{1-t}dt \approx \hat{\om}(r)=(1-t)W_1(t).
$$
Therefore, $W_1\in\R.$
By Lemma \ref{1210-3} and Theorem A, $\|\cdot\|_{A_\om^p}\approx\|\cdot\|_{A_{W_1}^p}$.
Then for all $p>0$, by Theorem 1 in \cite{hu3} we get
\begin{align*}
\|f\|_{A_\om^p}^p\approx |f(0)|^p+\int_\BB |\Re f(z)|^p(1-|z|)^p W_1(z)dV(z).
\end{align*}
For any $f\in H(\BB)$ and $|z|\leq \frac{1}{2}$,  let $f_r(z)=f(rz)$ for $r\in (0,1)$. By Cauchy's fomula, see \cite[Theorem 4.1]{Zk2005} for example,  we have
\begin{align*}
f(z)=f_\frac{3}{4}(\frac{4z}{3})=\int_\SS  \frac{f_{\frac{3}{4}}(\eta)}{(1-\langle \frac{4z}{3},\eta \rangle)^n}d\sigma(\eta).
\end{align*}
After a calculation, when $|z|\leq \frac{1}{2}$,
$$
|f(z)|\leq \|f\|_{A_\om^1},\,\,
|\Re f(z)|\lesssim |z|\|f_{\frac{3}{4}}\|_{H^\infty},\,\mbox{ and }\,
|\Re f(z)|\lesssim  |z|\|f\|_{A_\om^1}.
$$
We note that, when $|z|\geq \frac{1}{2}$, we have
$$\om^{n*}(z)=\int_{|z|}^1 t^{2n-1}\log\frac{t}{|z|}\om(t)dt\approx \int_{|z|} t^{1}\log\frac{t}{|z|}\om(t)dt=\om^*(z).$$
So, if $g\in\B$ and $f\in A_\om^1$, using  Lemma \ref{1210-3}, there exists a $C=C(n,\om,g)$, such that
\begin{align*}
|\langle f_r,g\rangle_{A_\om^2}|
&\leq C\left(\|f_r\|_{A_\om^1}+\|f_r\|_{A_\om^1}\int_{\frac{1}{2}\BB} \frac{\om^{n*}(z)}{|z|^{2n-2}}dV(z) +\int_{\BB\backslash \frac{1}{2}\BB} \frac{|\Re f_r(z){\Re g(z)}|}{|z|^{2n}}\om^{n*}(z)dV(z) \right) \\
& \approx \|f_r\|_{A_\om^1}+ \int_{\BB\backslash \frac{1}{2}\BB} {|\Re f_r(z){\Re g(z)}|}(1-|z|)\hat{\om}(z)dV(z) \\
& \leq  \|f_r\|_{A_\om^1}+ \|g\|_\B \int_{\BB} {|\Re f_r(z)|}\hat{\om}(z)dV(z) \\
& \approx  \|f_r\|_{A_\om^1}+ \|g\|_\B \int_{\BB} {|\Re f_r(z)|}(1-|z|)W_1(z)dV(z) \\
&\lesssim \|f\|_{A_\om^1}+\|g\|_\B \|f\|_{A_\om^1}.
\end{align*}
Therefore,  $g\in \B$ induce an element $F_g$ in $(A_\om^1)^*$  by the formula
$F_g(f)=\lim\limits_{r\to 1}\langle f_r,g\rangle_{A_\om^2}$ for all $f\in A_\om^1$.

On the other hand, the Hahn-Banach theorem and the well known fact (see \cite[Theorem 1.1]{zhu} for example) that
$$(L^1(\BB,\om dV))^*\simeq L^\infty(\BB,\om dV)$$
guarantee the existence of $\vp\in L^\infty$ such that
\begin{align*}
\lim\limits_{r\to 1}\langle f_r,g\rangle_{A_\om^2}=
F_g(f)=\int_\BB f(z)\ol{\vp(z)}\om(z)dV(z)
=\lim_{r\to 1}\int_\BB f_r(z)\ol{\vp(z)}\om(z)dV(z)
\end{align*}
for all $f\in A_\om^1$. Since $P_\om$ is self-adjoint and $P_\om(f_r)=f_r$, we have
$$
\int_\BB f_r(z)\ol{\vp(z)}\om(z)dV(z)
=\int_\BB P_\om(f_r)(z)\ol{\vp(z)}\om(z)dV(z)
=\int_\BB f_r(z)\ol{P_\om(\vp)(z)}\om(z)dV(z).
$$
By the first part of the proof, $P_\om\vp\in\B$.
Thus, $g-P_\om\vp\in\B$ and represents the zero functional. So, $g=P_\om\vp$.
The proof is complete.
\end{proof}

\begin{Remark}
By the above proof, if $\om\in\hD$, then $P_\om:L^\infty\to\B$ is bounded.
\end{Remark}

\begin{Theorem}\label{0413-4}
Suppose $1<p<\infty$ and $\om,\upsilon\in\dD$. Let $q=\frac{p}{p-1}$. Then the following statements are equivalent:
\begin{enumerate}[(i)]
  \item $P_\om^+: L_\upsilon^p\to  L_\upsilon^p$ is bounded;
  \item $P_\om: L_\upsilon^p\to  L_\upsilon^p$ is bounded;
  \item $M:=\sup\limits_{0\leq r<1} \frac{\hat{\upsilon}(r)^{\frac{1}{p}}}{\hat{\om}(r)}
  \left(\int_r^1 \frac{\om(s)^q}{\upsilon(s)^{q-1}}s^{2n-1}ds\right)^\frac{1}{q}<\infty;$
  \item $N:=\sup\limits_{0\leq r <1} \left(\int_0^r \frac{\upsilon(s)}{\hat{\om}(s)^p}s^{2n-1}ds+1\right)^\frac{1}{p}\left(\int_r^1 \frac{\om(s)^q}{\upsilon(s)^{q-1}}s^{2n-1}ds\right)^\frac{1}{q}<\infty.$
\end{enumerate}
\end{Theorem}

\begin{proof} When $n=1$, the theorem was proved in \cite{PjRj2019arxiv}. So, we always assume that $n\geq 2$.

$(i)\Rightarrow(ii)$. It is obvious.

$(ii)\Rightarrow(iii)$.  Suppose that $(ii)$ holds.
Let $P_\om^*$ be the adjoint of $P_\om$ with respect to $\langle\cdot,\cdot\rangle_{L_\upsilon^2}$. For all $f,g\in L^\infty$, by Fubini's Theorem,  we have
\begin{align*}
\langle f, P_\om^*g\rangle_{L_\upsilon^2}
&=\langle P_\om f,g\rangle_{L_\upsilon^2}  =\int_{\BB} P_\om f(z)\ol{g(z)}\upsilon(z)dV(z)\\
&=\int_\BB \left(\int_\BB f(\xi)\ol{B_z^\om(\xi)}\om(\xi)dV(\xi)\right) \ol{g(z)}\upsilon(z)dV(z)  \\
&=\int_\BB \left(\int_\BB \ol{g(z)B_z^\om(\xi)}\upsilon(z)dV(z)\right) f(\xi)\om(\xi)dV(\xi)\\
&=\int_\BB \left(\ol{\frac{\om(\xi)}{\upsilon(\xi)}\int_\BB {g(z)B_z^\om(\xi)}\upsilon(z)dV(z)}\right) f(\xi)\upsilon(\xi)dV(\xi).
\end{align*}
Since $L^\infty$ is dense in $L^p_\upsilon$ and $L_\upsilon^q$, then we have
\begin{align}\label{0514-1}
P_\om^*(g)=\frac{\om(\xi)}{\upsilon(\xi)}\int_\BB {g(z)B_z^\om(\xi)}\upsilon(z)dV(z),\,\,\,g\in  L_\upsilon^q.
\end{align}
By hypothesis, $P_\om^*$ is bounded on $L_\upsilon^q$.
Let $g_j(z)=z_1^j$, where $z=(z_1,z_2,\cdots,z_n)$ and $j\in\N\cup\{0\}$.
By  Lemma 1.11 in \cite{Zk2005} and Lemma \ref{0313-1},  we have
\begin{align*}
P_\om^*(g_j)(\xi)
&=\frac{\om(\xi)}{\upsilon(\xi)}\int_\BB {g_j(z)B_z^\om(\xi)}\upsilon(z)dV(z)   \\
&=\frac{1}{2n!}\frac{\om(\xi)}{\upsilon(\xi)}\sum_{k=0}^\infty \frac{(n-1+k)!}{k!\om_{2n+2k-1}} \int_\BB g_j(z) \langle \xi,z\rangle^k\upsilon(z)dV(z)   \\
&=\frac{2n}{2n!}\frac{\om(\xi)}{\upsilon(\xi)}\sum_{k=0}^\infty \frac{(n-1+k)!}{k!\om_{2n+2k-1}}
\int_0^1 r^{2n+k+j-1} \upsilon(r)dr\int_\SS \eta_1^j \langle \xi,\eta\rangle^kd\sigma(\eta)  \\
&=\xi_1^j\frac{\om(\xi)}{\upsilon(\xi)} \frac{\upsilon_{2n+2j-1}}{\om_{2n+2j-1}}\frac{(n-1+j)!}{j!(2n-1)!}\frac{(n-1)!j!}{(n-1+j)!}\\
&=\xi_1^j\frac{\om(\xi)}{\upsilon(\xi)} \frac{\upsilon_{2n+2j-1}}{\om_{2n+2j-1}}\frac{(n-1)!}{(2n-1)!}.
\end{align*}
By Lemmas \ref{0507-1} and \ref{1210-3}, we obtain
\begin{align*}
\|g_j\|_{L_\upsilon^q}^q
&=\int_{\BB}|z_1|^{qj}\upsilon(z)dV(z)
=2n\int_0^1 r^{2n+qj-1}\upsilon(r)dr\int_{\SS}|\eta_1|^{qj}d\sigma(\eta)  \\
&=2n {\upsilon}_{2n+qj-1} \int_{\SS}|\eta_1|^{qj}d\sigma(\eta)\\
&\approx {\upsilon}_{2n+2j-1} \int_{\SS}|\eta_1|^{qj}d\sigma(\eta),
\end{align*}
and
\begin{align*}
\|P_\om^*(g_j)(\xi)\|_{L_\upsilon^q}^q
&=\left(\frac{(n-1)!}{(2n-1)!}\frac{\upsilon_{2n+2j-1}}{\om_{2n+2j-1}}\right)^q\int_{\BB} |\xi_1|^{jq}\frac{\om^q(\xi)}{\upsilon^{q-1}(\xi)} dV(\xi)  \\
&\approx \left(\frac{\upsilon_{2n+2j-1}}{\om_{2n+2j-1}}\right)^q\int_0^1 r^{2n+qj-1}\frac{\om^q(r)}{\upsilon^{q-1}(r)}dr\int_\SS |\eta_1|^{jq} d\sigma(\eta)  \\
&\gtrsim \|g_j\|_{L_\upsilon^q}^q\frac{\upsilon_{2n+2j-1}^{q-1}}{\om_{2n+2j-1}^q}\int_{1-\frac{1}{2j+1}}^1 \frac{\om^q(r)}{\upsilon^{q-1}(r)}r^{2n-1}dr \\
&\approx   \|g_j\|_{L_\upsilon^q}^q\frac{\upsilon_{2j+1}^{q-1}}{\om_{2j+1}^q}\int_{1-\frac{1}{2j+1}}^1 \frac{\om^q(r)}{\upsilon^{q-1}(r)}r^{2n-1}dr.
\end{align*}
Let $r_j=1-\frac{1}{2j+1}$. We get
\begin{align*}
\|P_\om^*(g_j)(\xi)\|_{L_\upsilon^q}^q
\gtrsim \|g_j\|_{L_\upsilon^q}^q\frac{\hat{\upsilon}(r_j)^{q-1}}{\hat{\om}(r_j)^q}\int_{r_j}^1 \frac{\om^q(r)}{\upsilon^{q-1}(r)}r^{2n-1}dr .
\end{align*}
Let
$$H(t)=\frac{\hat{\upsilon}(t)^{q-1}}{\hat{\om}(t)^q}\int_{t}^1 \frac{\om^q(r)}{\upsilon^{q-1}(r)}r^{2n-1}dr.$$
When $r_j\leq t< r_{j+1}$, we have
$H(t)\lesssim H(r_j)$.
Thus, by the assumption, we have get $\sup_{t\geq 0}H(t)<\infty$   as desired.

$(iii)\Rightarrow(i)$.  Suppose that $(iii)$ holds. For all $z\in\BB$, let
\begin{align*}
h(z)=\upsilon(z)^\frac{1}{p}\left(\int_{|z|}^1 \frac{\om(s)^{q}}{\upsilon(s)^{q-1}}s^{2n-1}ds\right)^\frac{1}{pq}.
\end{align*}
By the hypothesis, we have
\begin{align}\label{0515-2}
\int_t^1 \left(\frac{\om(s)}{h(s)}\right)^q s^{2n-1}ds
= q\left(\int_t^1 \frac{\om(s)^q}{\upsilon(s)^{q-1}} s^{2n-1}ds\right)^\frac{1}{q}
\lesssim M\frac{\hat{\om}(t)}{\hat{\upsilon}(t)^\frac{1}{p}}.
\end{align}
If $r|z|\leq \frac{1}{4}$, by Lemma \ref{0313-1}, we have
$$M_1(r,B_z^\om)\leq \|B_{rz}^\om\|_{H^\infty}\approx \frac{1}{\hat{\om}(S_{rz})}\approx 1.$$
If $r|z|>\frac{1}{4}$, by Lemma \ref{0507-1}, there exists a constant $a>0$  such that $\frac{\hat{\om}(t)}{(1-t)^a}$ is essentially decreasing. Then by Lemma \ref{0313-2}, we have
$$M_1(r,B_z^\om)\lesssim \int_0^{r|z|}\frac{dt}{\hat{\om}(t)(1-t)}\lesssim \frac{(1-r|z|)^a}{\hat{\om}(r|z|)}\int_0^{r|z|}\frac{dt}{(1-t)^{a+1}}\approx \frac{1}{\hat{\om}(r|z|)}.$$
So, for all $r\in(0,1)$ and $z\in\BB$,  we have
\begin{align}\label{0515-4}
M_1(r,B_z^\om)\lesssim 1+ \int_0^{r|z|}\frac{1}{\hat{\om}(t)(1-t)}dt\lesssim  \frac{1}{\hat{\om}(r|z|)}.
\end{align}
Hence, by (\ref{0515-2}), (\ref{0515-4}), Fubini's theorem and Lemma \ref{0507-1}, we obtain
\begin{align*}
\int_\BB |B_z^\om(\xi)|\left(\frac{\om(\xi)}{h(\xi)}\right)^qdV(\xi)
=&2n \int_0^1 \left(\frac{\om(r)}{h(r)}\right)^q r^{2n-1} M_1(r, B_z^\om)dr \\
\lesssim & \int_0^1 \left(\frac{\om(r)}{h(r)}\right)^q r^{2n-1}\left(1+ \int_0^{r|z|}\frac{1}{\hat{\om}(t)(1-t)}dt\right) dr   \\
\lesssim & M\frac{\hat{\om}(0)}{\hat{\upsilon}(0)^\frac{1}{p}}+\int_0^{|z|}\frac{1}{\hat{\om}(t)(1-t)}\int_{\frac{t}{|z|}}^1 \left(\frac{\om(r)}{h(r)}\right)^q r^{2n-1} drdt  \\
\leq & M\frac{\hat{\om}(0)}{\hat{\upsilon}(0)^\frac{1}{p}}+\int_0^{|z|}\frac{1}{\hat{\om}(t)(1-t)}\int_{t}^1 \left(\frac{\om(r)}{h(r)}\right)^q r^{2n-1} drdt  \\
\lesssim & M+M \int_0^{|z|}\frac{1}{\hat{\upsilon}(t)^\frac{1}{p}(1-t)}dt
\lesssim \frac{M}{\hat{\upsilon}(|z|)^\frac{1}{p}}.
\end{align*}
Therefore, by   H\"older's inequality  and Fubini's theorem  we get
\begin{align}
\|P_\om^+(f)\|_{L_\upsilon^p}^p
&=\int_\BB \upsilon(z)\left| \int_\BB f(\xi) |B_z^\om(\xi)|\om(\xi)dV(\xi) \right|^pdV(z)  \nonumber\\
&\leq \int_\BB\left(\int_\BB |f(\xi)|^p h(\xi)^p |B_z^\om(\xi)|dV(\xi)\right)
\left(\int_\BB |B_z^\om(\xi)|\left(\frac{\om(\xi)}{h(\xi)}\right)^qdV(\xi)\right)^\frac{p}{q}\upsilon(z)dV(z)  \nonumber \\
&\lesssim M^\frac{p}{q}\int_\BB\left(\int_\BB |f(\xi)|^p h(\xi)^p |B_z^\om(\xi)|dV(\xi)\right)
\frac{\upsilon(z)}{\hat{\upsilon}(z)^\frac{1}{q}}dV(z)   \nonumber\\
&=M^\frac{p}{q}\int_\BB|f(\xi)|^p h(\xi)^p\left(\int_\BB  |B_z^\om(\xi)|\frac{\upsilon(z)}{\hat{\upsilon}(z)^\frac{1}{q}}dV(z) \right) dV(\xi).\label{0618-1}
\end{align}
Since $|B_z^\om(\xi)|=|B_\xi^\om(z)|$, by (\ref{0515-4}),  we have
\begin{align}
\int_{\BB\backslash |\xi|\BB} |B_z^\om(\xi)|\frac{\upsilon(z)}{\hat{\upsilon}(z)^\frac{1}{q}}dV(z)
&\lesssim \int_{|\xi|}^1 \frac{\upsilon(r)}{\hat{\upsilon}(r)^\frac{1}{q}}M_1(r,B_\xi^\om)dr
\lesssim \frac{\hat{\upsilon}(\xi)^\frac{1}{p}}{\hat{\om}(\xi)}, \label{0322-1}
\end{align}
and
\begin{align}
\int_{ |\xi|\BB} |B_z^\om(\xi)|\frac{\upsilon(z)}{\hat{\upsilon}(z)^\frac{1}{q}}dV(z)
\lesssim \int_0^{|\xi|}\frac{\upsilon(r)}{\hat{\upsilon}(r)^\frac{1}{q}}\frac{1}{\hat{\om}(r|\xi|)}r^{2n-1}dr
\lesssim \int_0^{|\xi|} \frac{\upsilon(r)}{\hat{\upsilon}(r)^\frac{1}{q}\hat{\om}(r)}r^{2n-1}dr. \label{0322-2}
\end{align}
By the assumption, we have
\begin{align}\label{0518-1}
\int_0^1\frac{\om(s)^q}{\upsilon(s)^{q-1}}s^{2n-1}ds<\infty,\,\,
\int_0^{{\frac{1}{2}}}  \frac{\upsilon(t)}{\hat{\om}(t)^p} t^{2n-1}dt>0,\,\,
\int_{\frac{1}{2}}^1  \frac{\upsilon(t)}{\hat{\om}(t)^p} t^{2n-1}dt>0.
\end{align}

Then when $r\leq {\frac{1}{2}}$,  we have
$$ \hat{\om}(r)\approx 1\approx
\hat{\upsilon}(r)^\frac{1}{p}\left(\int_r^{1}  \frac{\om(t)^q}{\upsilon(t)^{q-1}} t^{2n-1}dt\right)^\frac{1}{q}. $$
When $r> {\frac{1}{2}}$, by H\"older's inequality, we have
$$\hat{\om}(r)=\int_r^1\om(t)dt
\leq\hat{\upsilon}(r)^\frac{1}{p}\left(\int_r^1 \frac{\om(t)^q}{\upsilon(t)^{q-1}}dt\right)^\frac{1}{q}
\approx \hat{\upsilon}(r)^\frac{1}{p}\left(\int_r^1 \frac{\om(t)^q}{\upsilon(t)^{q-1}}t^{2n-1}dt\right)^\frac{1}{q}.$$
Then, for all $r\in(0,1)$, we have
$$
\frac{\hat{\om}(r)^p}{\hat{\upsilon}(r)}\int_0^r \frac{\upsilon(t)}{\hat{\om}(t)^p}t^{2n-1}dt
\lesssim
\left(\int_r^1 \frac{\om(t)^q}{\upsilon(t)^{q-1}}t^{2n-1}dt\right)^\frac{p}{q}\int_0^r \frac{\upsilon(t)}{\hat{\om}(t)^p} t^{2n-1}dt.
$$
We claim
\begin{align}
K_*:=\sup_{0\leq r<1} \left(\int_r^1 \frac{\om(t)^q}{\upsilon(t)^{q-1}}t^{2n-1}dt\right)^\frac{1}{q}
\left(\int_0^r \frac{\upsilon(t)}{\hat{\om}(t)^p}t^{2n-1}dt\right)^\frac{1}{p}<\infty.    \label{0323-1}
\end{align}
Taking this for granted for a moment. Then,
\begin{align}
 \int_0^{|\xi|} \frac{\upsilon(r)}{\hat{\upsilon}(r)^\frac{1}{q}\hat{\om}(r)}r^{2n-1}dr
 &\leq \int_0^{|\xi|} \frac{\upsilon(r)}{\hat{\om}(r)} \left(\frac{K_*^p}{\hat{\om}(r)^p\int_0^r \frac{\upsilon(t)}{\hat{\om}(t)^p}t^{2n-1}dt}\right)^\frac{1}{q}r^{2n-1}dr   \nonumber\\
 &=K_*^{p-1}\int_0^{|\xi|} \frac{\upsilon(r)}{\hat{\om}(r)^p} \left(\int_0^r \frac{\upsilon(t)}{\hat{\om}(t)^p}t^{2n-1}dt\right)^{-\frac{1}{q}}r^{2n-1}dr \nonumber\\
 &\approx K_*^{p-1} \left(\int_0^{|\xi|} \frac{\upsilon(t)}{\hat{\om}(t)^p}t^{2n-1}dt\right)^{\frac{1}{p}}.\label{0323-2}
\end{align}
By (\ref{0322-1}), (\ref{0322-2}),  (\ref{0323-1}) and (\ref{0323-2}), we have
\begin{align}\label{0618-2}
h(\xi)^p\int_{\BB\backslash |\xi|\BB} |B_z^\om(\xi)|\frac{\upsilon(z)}{\hat{\upsilon}(z)^\frac{1}{q}}dV(z)
\lesssim  \upsilon(\xi)
 \left(\int_{|\xi|}^1 \frac{\om(t)^q}{\upsilon(t)^{q-1}}t^{2n-1}dt\right)^\frac{1}{q}
 \frac{\hat{\upsilon}(\xi)^\frac{1}{p}}{\hat{\om}(\xi)}
 \leq M \upsilon(\xi),
\end{align}
and
\begin{align}
h(\xi)^p\int_{ |\xi|\BB} |B_z^\om(\xi)|\frac{\upsilon(z)}{\hat{\upsilon}(z)^\frac{1}{q}}dV(z)
&\lesssim K_*^{p-1} \upsilon(\xi)
 \left(\int_{|\xi|}^1 \frac{\om(t)^q}{\upsilon(t)^{q-1}}t^{2n-1}dt\right)^\frac{1}{q}
 \left(\int_0^{|\xi|} \frac{\upsilon(t)}{\hat{\om}(t)^p}t^{2n-1}dt\right)^{\frac{1}{p}}  \nonumber\\
& \leq K_*^p\upsilon(\xi). \label{0618-3}
\end{align}
So, by (\ref{0618-1}), (\ref{0618-2}) an (\ref{0618-3}), we obtain
\begin{align*}
\|P_\om^+(f)\|_{L_\upsilon^p}^p
&\lesssim  \int_\BB|f(\xi)|^p \upsilon(\xi) dV(\xi)=\|f\|_{L_\upsilon^p}^p.
\end{align*}

Now, we  prove that (\ref{0323-1}) holds. Assume $r>\frac{1}{2}$.  An integration by parts and H\"older's inequality give
\begin{align*}
\int_0^r \frac{\upsilon(t)}{\hat{\om}(t)^p}t^{2n-1}dt
\leq & \int_0^\frac{1}{2} \frac{\upsilon(t)}{\hat{\om}(t)^p}dt +\int_\frac{1}{2}^r \frac{\upsilon(t)}{\hat{\om}(t)^p}dt \\
\lesssim&  1+\int_\frac{1}{2}^r \frac{\hat{\upsilon}(t)}{\hat{\om}(t)^p} \frac{\om(t)}{\upsilon(t)^\frac{1}{p}} \frac{\upsilon(t)^\frac{1}{p}}{\hat{\om}(t)}t^{2n-1} dt    \\
\leq& 1+
\left(\int_0^r \left(\frac{\hat{\upsilon}(t)}{\hat{\om}(t)^p} \frac{\om(t)}{\upsilon(t)^\frac{1}{p}} \right)^q t^{2n-1}dt\right)^\frac{1}{q}
\left(\int_0^r  \frac{\upsilon(t)}{\hat{\om}(t)^p} t^{2n-1}dt\right)^\frac{1}{p}                               \\
=& 1+  J_1^\frac{1}{q}
\left(\int_0^r  \frac{\upsilon(t)}{\hat{\om}(t)^p} t^{2n-1}dt\right)^\frac{1}{p},
\end{align*}
where
$$J_1=\int_0^r \left(\frac{\hat{\upsilon}(t)}{\hat{\om}(t)^p} \frac{\om(t)}{\upsilon(t)^\frac{1}{p}} \right)^q t^{2n-1}dt.$$
Since
\begin{align*}
J_1&=\int_0^r \left(\frac{\hat{\upsilon}(t)^\frac{1}{p}}{\hat{\om}(t)}\left(\int_t^1 \frac{\om(s)^q}{\upsilon(s)^{q-1}}s^{2n-1}ds\right)^{\frac{1}{q}}\right)^{pq}
\frac{ \frac{\om(t)^q}{\upsilon(t)^{q-1}} }{\left(\int_t^1 \frac{\om(s)^q}{\upsilon(s)^{q-1}}s^{2n-1}ds\right)^p}t^{2n-1} dt \\
&\lesssim  \frac{M^{pq} }{\left(\int_r^1 \frac{\om(s)^q}{\upsilon(s)^{q-1}}s^{2n-1}ds\right)^{p-1}},
\end{align*}
we obtain
\begin{align*}
\int_0^r \frac{\upsilon(t)}{\hat{\om}(t)^p}t^{2n-1}dt
\lesssim 1+ M^p
{\left(\int_0^r  \frac{\upsilon(t)}{\hat{\om}(t)^p} t^{2n-1}dt\right)^\frac{1}{p} }
{\left(\int_r^1 \frac{\om(s)^q}{\upsilon(s)^{q-1}}s^{2n-1}ds\right)^{-\frac{p}{q^2}}}.
\end{align*}
Then
\begin{align*}
\left(\int_0^r \frac{\upsilon(t)}{\hat{\om}(t)^p}t^{2n-1}dt\right)^\frac{1}{p}
\lesssim 1+ M
{\left(\int_0^r  \frac{\upsilon(t)}{\hat{\om}(t)^p} t^{2n-1}dt\right)^\frac{1}{p^2} }
{\left(\int_r^1 \frac{\om(s)^q}{\upsilon(s)^{q-1}}s^{2n-1}ds\right)^{-\frac{1}{q^2}}}.
\end{align*}
Multiply   the expression by  $\left(\int_r^1 \frac{\om(s)^q}{\upsilon(s)^{q-1}}s^{2n-1}ds\right)^{\frac{1}{q}}$, we have
\begin{align*}
J_2(r)
\lesssim & \left(\int_r^1 \frac{\om(s)^q}{\upsilon(s)^{q-1}}s^{2n-1}ds\right)^{\frac{1}{q}} +MJ_2(r)^\frac{1}{p},
\end{align*}
where,
$$J_2(r)=\left(\int_r^1 \frac{\om(s)^q}{\upsilon(s)^{q-1}}s^{2n-1}ds\right)^{\frac{1}{q}}\left(\int_0^r \frac{\upsilon(t)}{\hat{\om}(t)^p}t^{2n-1}dt\right)^\frac{1}{p}.$$
Using (\ref{0518-1}), we have
\begin{align*}
J_2(r)^\frac{1}{q}
\lesssim {\left(\int_r^1 \frac{\om(s)^q}{\upsilon(s)^{q-1}}s^{2n-1}ds\right)^\frac{1}{q^2}}{\left(\int_0^r  \frac{\upsilon(t)}{\hat{\om}(t)^p} t^{2n-1}dt\right)^{-\frac{1}{p^2}}}
+M<\infty.
\end{align*}
Therefore,
$$\sup_{r>\frac{1}{2}} \left(\int_r^1 \frac{\om(s)^q}{\upsilon(s)^{q-1}}s^{2n-1}ds\right)^{\frac{1}{q}}\left(\int_0^r \frac{\upsilon(t)}{\hat{\om}(t)^p}t^{2n-1}dt\right)^\frac{1}{p}   <\infty.$$
When $r\leq \frac{1}{2}$,  (\ref{0323-1}) holds obviously.

$(iii)\Rightarrow(iv)$. Using (\ref{0518-1}) and (\ref{0323-1}), we can get the desired result.

$(iv)\Rightarrow(iii)$.  Suppose that $(iv)$ holds, that is,
$$N:=\sup\limits_{0\leq r <1} \left(\int_0^r \frac{\upsilon(s)}{\hat{\om}(s)^p}s^{2n-1}ds+1\right)^\frac{1}{p}\left(\int_r^1 \frac{\om(s)^q}{\upsilon(s)^{q-1}}s^{2n-1}ds\right)^\frac{1}{q}<\infty.$$

Since $\om\in\dD$, by Lemma \ref{0507-1}, there exists $b>0$ such that $\frac{\hat{\om}(r)^p}{(1-r)^b}$ is essentially increasing.
Then
$$\int_0^r \frac{\upsilon(s)}{\hat{\om}(s)^p}s^{2n-1}ds\gtrsim \frac{(1-r)^b}{\hat{\om}(r)^p}\int_0^r \frac{\upsilon(s)}{(1-s)^b}s^{2n-1}ds.$$
Since $\upsilon\in\dD$, there exist $C>1$ and $K>1$ such that
$$\hat{\upsilon}(r)\geq C\hat{\upsilon}(1-\frac{1-r}{K}).$$
Let $r_k=1-K^{-k}$, $k=0,1,2,\cdots.$
For any $r_2\leq r<1$, there is an integer  $x=x(r)$ such that $r_x\leq r<r_{x+1}$. Then
\begin{align*}
(1-r)^b\int_0^r \frac{\upsilon(s)}{(1-s)^b}s^{2n-1}ds
\geq& \sum_{k=0}^{x-1}\int_{r_{k}}^{r_{k+1}} \left(\frac{1-r}{1-s}\right)^b\upsilon(s)s^{2n-1}ds\\
\geq&  \sum_{k=0}^{x-1}r_k^{2n-1}\left(\frac{1-r_{x+1}}{1-r_k}\right)^b\left(\hat{\upsilon}(r_k)-\hat{\upsilon}(r_{k+1})\right)\\
\geq&\sum_{k=0}^{x-1}r_k^{2n-1}\frac{C-1}{CK^{(x+1-k)b}}\hat{\upsilon}(r_k)\\
\geq&\sum_{k=0}^{x-1}r_k^{2n-1}\frac{(C-1)C^{x-1-k}}{K^{(x+1-k)b}}\hat{\upsilon}(r_x)\\
\geq& \hat{\upsilon}(r)\frac{C-1}{C^2}\sum_{s=2}^{x+1} r_{x+1-s}^{2n-1}\left(\frac{C}{K^b}\right)^s \\
\geq& r_{x-1}^{2n-1}\hat{\upsilon}(r)\frac{C-1}{K^{2b}}
\geq r_1^{2n-1}\hat{\upsilon}(r)\frac{C-1}{K^{2b}}.
\end{align*}
So, when $r\geq r_2$, we have
$$\int_0^r \frac{\upsilon(s)}{\hat{\om}(s)^p}s^{2n-1}ds\gtrsim \frac{\hat{\upsilon}(r)}{\hat{\om}(r)^p}.$$
Therefore,
$$\sup\limits_{r_2\leq r<1} \frac{\hat{\upsilon}(r)^{\frac{1}{p}}}{\hat{\om}(r)}
  \left(\int_r^1 \frac{\om(s)^q}{\upsilon(s)^{q-1}}s^{2n-1}ds\right)^\frac{1}{q}<\infty.$$
 When $r< r_2$,   $(iii)$ holds obviously.  The proof is complete.
\end{proof}



\noindent {\bf Acknowledgments.} 
  The corresponding author was supported by  NNSF of China (No.11720101003) and the Macao Science and Technology Development Fund (No. 186/2017/A3).

\end{document}